\documentclass[12pt]{article}
\usepackage{fancyhdr}
\usepackage{graphicx,wrapfig}
\usepackage{subfigure}
\usepackage{caption}
\usepackage{amsmath,amsthm}
\usepackage{amssymb}
\usepackage{makecell}
\usepackage{multirow}
\usepackage{array}
\usepackage{setspace}
\usepackage{float}
\usepackage[dvipsnames]{xcolor}
\usepackage{color,mathtools}
\usepackage{enumerate}
\usepackage{epstopdf}
\usepackage{lipsum}
\usepackage{appendix}
\usepackage{hyperref}
\usepackage{algorithm}
\usepackage{algorithmic}
\usepackage{blindtext}

\newcommand{\bs}[1]{\boldsymbol{#1}}
\newcommand{\bmu}{\bs{\mu}}

\newcommand{\calW}{\mathcal{W}}
\newcommand{\calR}{\mathcal{R}}
\newcommand{\calN}{\mathcal{N}}
\newcommand{\bbR}{\mathbb{R}}
\newcommand{\calD}{\mathcal{D}}
\newcommand{\N}{{\cal N}}

\DeclareMathOperator*{\argmax}{arg\,max}

\newtheorem{theorem}{Theorem}
\newtheorem{corollary}{Corollary}

\newtheorem{remark}{Remark}[section]
\newtheorem{lemma}{Lemma}[section]

\title{A reduced basis warm-start iterative solver for the parameterized linear systems}
\author{Shijin Hou\thanks{School of Mathematical Sciences, University of Science and Technology of China, Hefei, Anhui
230026, People's Republic of China. ({\tt houshiji@mail.ustc.edu.cn}).}
\and Yanlai Chen\thanks{Department of Mathematics, University of Massachusetts Dartmouth, 285 Old Westport Road, North Dartmouth, MA 02747, USA. ({\tt yanlai.chen@umassd.edu}). The work of this author was partially supported by National Science Foundation grant DMS-2208277, by Air Force Office of Scientific Research grant FA9550-23-1-0037, and by the UMass Dartmouth Marine and UnderSea Technology (MUST) Research Program made possible via an Office of Naval Research grant N00014-22-1-2012.}
\and Yinhua Xia\thanks{School of Mathematical Sciences, University of Science and Technology of China, Hefei, Anhui
	230026, People's Republic of China. ({\tt yhxia@ustc.edu.cn}). The work of this author was partially supported by National Key R\&D Program of China No. 2022YFA1005202/ 2022YFA1005200, and by National Natural Science Foundation of China grant 12271498.}}
\date{}

\begin{document}

\maketitle

\begin{abstract}
This paper proposes and tests the first-ever reduced basis warm-start iterative method for the parametrized linear systems, exemplified by those discretizing the parametric partial differential equations. Traditional iterative methods are usually used to obtain the high-fidelity solutions of these linear systems. However, they typically come with a significant computational cost which becomes challenging if not entirely untenable when the parametrized systems need to be solved a large number of times (e.g. corresponding to different parameter values or time steps). Classical techniques for mitigating this cost mainly include acceleration approaches such as preconditioning. This paper advocates for the generation of an initial prediction with controllable fidelity as an alternative approach to achieve the same goal. The proposed reduced basis warm-start iterative method leverages the mathematically rigorous and efficient reduced basis method to generate a high-quality initial guess thereby decreasing the number of iterative steps. Via comparison with the iterative method initialized with a zero solution and the RBM preconditioned and initialized iterative method tested on two 3D steady-state diffusion equations, we establish the efficacy of the proposed reduced basis warm-start approach.
\end{abstract}

\section{Introduction}
In this work, we consider the parametrized linear systems that take the general form of
\begin{equation}\label{discrete equation}
A_h(\bmu)u_h(\bmu) = f_h(\bmu),
\end{equation}
where $A_h(\bmu)$ denotes a parameter-dependent matrix of dimension $\N\times \N$, $u_n(\bmu),\ f_h(\bmu)\in 
\bbR^{\N}$ are $\N$-dimensional vectors, and $\bmu\in \mathcal{D}\subset\mathbb{R}^{p}$ is a $p$-dimensional parameter vector. They often stem from the discretization of parametric partial differential equations (pPDEs) whose real-time solution is widely in demand for many applications such as optimization, uncertainty quantification, and inverse problems. 
These systems can usually be solved with high precision by various iterative methods \cite{kelley1995iterative}, especially when the system is large. Several classical iterative methods, such as Jacobi, Richardson, and Gauss-Seidel methods, etc, are the simplest options. They are not only used as standalone solvers but also as preconditioners for accelerating other methods. Similarly, multigrid methods (MG) have been widely developed as both iterative methods \cite{brandt1977multi,hackbusch2013multi} and preconditioners \cite{tatebe1993multigrid}. There is another class of mainstream methods, Krylov subspace methods, which includes CG \cite{hestenes1952methods}, BiCGSTAB \cite{van1992bi}, GMRES \cite{saad1986gmres}, etc. 

However, all these high-fidelity iterative methods depend on the full-order model (FOM) of large degrees of freedom (DoFs). Some of these systems need to be solved toward machine precision (e.g. long-time simulations in astrophysics) which means that it will take a large number of iterative steps. Both factors contribute to an extremely time-consuming process. What exacerbates the situation is that repeatedly solving such problems for different parameter instances is often necessary. It is thus imperative to design efficient and reliable solvers for such problems that converge to machine precision.

In recent decades, the reduced basis method (RBM), as a class of reduced order modeling techniques, has been developed and widely used to obtain the fast solution of pPDEs \cite{Quarteroni2015,HesthavenRozzaStammBook,Haasdonk2017Review}. The RBM achieves high efficiency via an offline-online decomposition strategy and a mathematically rigorous procedure to build a surrogate solution space. In the offline phase, a reduced basis (RB) space, $W_N$, of dimension $N\ll\N$, the number of FOM DoFs, is successfully built by the greedy algorithm \cite{Rozza2008Reduced, Binev2011Convergence} 
or the proper orthogonal decomposition (POD) \cite{holmes_lumley_berkooz_1996, LIANG2002527}. With this surrogate approximate space, a reduced-order model (ROM) can be derived that enforces the PDE at the reduced level making the ROM physics-informed as opposed to purely data-driven. Subsequently, for different values of parameters, we only need to solve the ROM with a much lower computational cost in the online phase.

In this work, we propose and test the first-ever Reduced Basis Warm-Start (RBWS) iterative method leveraging the mature RBM framework to address the cost challenge of using the traditional iterative methods to repeatedly solve parametrized linear systems. The specific RBM we employ is a highly efficient variant, the so-called L1-based reduced over-collocation (L1ROC) method \cite{Chen2021L1Based}. After a learning stage with a bare-minimum overhead cost thanks to a cost-free L1-norm based error indicator \cite{chen2019robust}, it is capable of providing a highly accurate initial prediction for the iterative methods. In essence, we are developing a data-driven warm-start approach based on the state-of-the-art physics-informed model order reduction strategies for the traditional iterative solvers for linear systems.
Furthermore, we study the acceleration functionality of the RBM as a preconditioner. Via a formal analysis, we show that the approximation accuracy of the RB space of a fixed dimension, due to the need of preserving the computational efficiency, will deteriorate as the iteration goes on thus having a gradually more limited acceleration effect. 
In the numerical experiments, we implement the multigrid preconditioned conjugate gradient (MGCG) method \cite{tatebe1993multigrid} and the multispace reduced basis preconditioned conjugate gradient (MSRBCG) method \cite{Dal2018MultiSpaceReducedBasis} initialized with the RB initial value. The numerical results demonstrate that the MGCG method with our proposed warm-start approach has the best performance on both convergence and efficiency.

We remark that this is by far not the first attempt to hybridize a surrogate model with a full order iterative solver. There has been a class of hybrid methods, in which the data-driven models have been employed to improve the traditional iterative methods for saving computational effort or accelerating convergence. In \cite{zhou2023neural}, the authors proposed a neural network warm-start approach for solving the high-frequency inverse acoustic obstacle problem. A combination of deep feedforward neural networks and convolutional autoencoders is used to establish an approximate mapping from the parameter space to the solution space that serves as a means to obtain very accurate initial predictions in \cite{nikolopoulosa2022ai}. In \cite{kadeethum2022enhancing}, the non-intrusive reduced order model was used to improve the computational efficiency of the high-fidelity nonlinear solver. In \cite{Dal2018MultiSpaceReducedBasis,Dal2018Stokesequations,nguyen2018reduced,nikolopoulosa2022ai}, the authors proposed a multi-space reduced basis preconditioner by combining an iteration-dependent RB component that is derived from an RB solver and a fine preconditioner. Different preconditioned iterative methods such as the Richardson method and the Krylov subspace methods with this precondition have been studied. However, it's worth pointing out that these methods were not tested till machine accuracy which is the regime that a purely data-driven approach (such as the non-physics-informed neural network) or the RBM preconditioning may encounter challenges. In contrast, our proposed RBWS method does not face these challenges, and we formally analyze the limitations of the RBM preconditioning techniques.

The structure of this paper is as follows. Firstly, two iterative methods are reviewed in Section \ref{iterative solvers}. In Section \ref{RBM-initialized iterative solvers}, we present the new RBWS method and analyze the limitations of the RBM preconditioning techniques. The numerical results are shown in Section \ref{sec:numerical result} to demonstrate the convergence and the efficiency of the proposed method. Finally, concluding remarks are drawn in Section \ref{sec:Conclusion}.

\section{Background}
\label{iterative solvers}
This section is devoted to the review of two efficient preconditioned iterative methods for solving the linear system \eqref{discrete equation}, namely, the multigrid preconditioned conjugate gradient (MGCG) method \cite{tatebe1993multigrid} and the multispace reduced basis preconditioned conjugate gradient (MSRBCG) method \cite{Dal2018MultiSpaceReducedBasis}. To start, we briefly outline the conjugate gradient method and its preconditioned version since both MGCG and MSRBCG are specific cases of the preconditioned conjugate gradient (PCG) method.

\subsection{Preconditioned conjugate gradient method}
The conjugate gradient (CG) method \cite{hestenes1952methods} was designed for solving the symmetric positive-defined linear systems. The idea is to solve the equivalent optimization problem aiming to minimize the following quadratic function
\[
Q(u_h(\bmu)) = \frac{1}{2}u_h^{T}(\bmu)A_h(\bmu)u_h(\bmu)-f_h^T(\bmu)u_h(\bmu).
\]
The CG algorithm recursively solves for the $k^{\rm th}$ iteration  $u_h^{(k)}(\bmu)$ in the $k^{\rm th}$ Krylov subspaces $\mathcal{K}_k$,
\[
u_h^{(k)}(\bmu) = \arg\min_{u_h(\bmu)\in \mathcal{K}_k}Q(u_h(\bmu)).
\]
The PCG method is an enhanced CG method through preconditioning. Applying a linear transformation to the original system \eqref{discrete equation} with the matrix $T(\bmu)$, called the preconditioner, results in the following system
\begin{equation}\label{preconditioned discrete equation} 
T^{-1}(\bmu)A_h(\bmu)u_h(\bmu) = T^{-1}(\bmu)f_h(\bmu),
\end{equation}
which, with an appropriately designed $T(\bmu)$, features a smaller condition number than that of  \eqref{discrete equation}. Invoking the CG method to solve the preconditioned system \eqref{preconditioned discrete equation} leads to the PCG method for \eqref{discrete equation} as presented in Algorithm \ref{alg:PCG}. We denote the application of the  $k^{\rm th}$ PCG iteration to obtain the $k^{\rm th}$ solution $u_h^{(k)}(\bmu)$ with respect to the $(k-1)^{\rm th}$ solution $u_h^{(k-1)}(\bmu)$ by
\begin{equation}
u_h^{(k)}(\bmu) = {\rm PCG}(A_h(\bmu), f_h(\bmu), u_h^{(k-1)}(\bmu); {\mathcal P}(\cdot)).
\label{eq:pcg}
\end{equation}
To describe a generic framework including the MGCG and MSRBCG algorithms presented below, we denote the preconditioner by ${\mathcal P}(\cdot)$. The vanilla version above is nothing but PCG with ${\mathcal P}(\cdot)$ defined as matrix multiplication (linear solve) with $T^{-1}(\bmu)$.
\begin{algorithm}[htb]
\caption{PCG algorithm with a generic preconditioner ${\mathcal P}(\cdot)$}
\label{alg:PCG}
\begin{algorithmic}[1]
\STATE \textbf{Input:} $A_h(\bmu)\in \mathbb{R}^{\N\times\N}$, $f_h(\bmu)\in \mathbb{R}^{\N}$, preconditioner ${\mathcal P}(\cdot)$, the residual tolerance $\delta$, the maximum number of iterations $L_{\rm max}$ and an initial value $u_h^{(0)}(\bmu)$
\STATE Compute initial residual $r_h^{(0)}(\bmu) = f_h(\bmu)-A_h(\bmu)u_h^{(0)}(\bmu)$ and set $k=0$ 
\STATE $s_0(\bmu) = {\mathcal P}(r_h^{(0)}(\bmu))$ 
\STATE $p_0(\bmu) = s_0(\bmu)$
\WHILE{$\Vert r_h^{(k)}(\bmu)\Vert/\Vert f_h(\bmu)\Vert<\delta$ $\&$ $k< L_{\rm max}$}
\STATE $\alpha_k(\bmu) = \frac{(r_h^{(k)}(\bmu))^{T}s_k(\bmu)}{p_k^{T}(\bmu)A_h(\bmu)p_k(\bmu)}$
\STATE $u_h^{(k+1)}(\bmu) = u_h^{(k)}(\bmu)+\alpha_k(\bmu)p_k(\bmu)$
\STATE $r_h^{(k+1)}(\bmu) = r_h^{(k)}(\bmu)-\alpha_k(\bmu)A_h(\bmu)p_k(\bmu)$
\STATE $s_{k+1}(\bmu) = {\mathcal P}(r_h^{(k+1)}(\bmu))$ \label{PCG_precondition} 
\STATE $\beta_k(\bmu) = \frac{(r_h^{(k+1)}(\bmu))^{T}s_{k+1}(\bmu)}{(r_h^{(k)}(\bmu))^Ts_k(\bmu)}$
\STATE $p_{k+1}(\bmu)=s_{k+1}(\bmu)+\beta_k(\bmu)p_k(\bmu)$
\STATE $k = k+1$
\ENDWHILE
\end{algorithmic}
\end{algorithm}

\subsection{Multigrid preconditioner}\label{Sec:MGCG}
The multigrid (MG) method \cite{brandt1977multi,hackbusch2013multi} exploits the coarse-grid correction to overcome the limitation of the classical iterative approaches that tend to efficiently eliminate the high-frequency error, but not the low-frequency one. We consider an MG method with $J+1$ ($J\ge 1$) levels. Let $P_i(\bmu)$ denotes the prolongation operator from level $i$ to level $i+1$ with $0\le i \le J$. With the finest coefficient matrix $A_h^J(\bmu) = A_h(\bmu)$, then the coefficient matrix $A_h^i(\bmu)$ at $i^{\rm th}$ level grid can be computed by $A_h^i(\bmu) = P_i^T(\bmu)A_h^{i+1}(\bmu)P_i(\bmu)$ for $0\le i\le J-1$. Considering a general equation $Au=b$, we denote 
\begin{equation}\label{smoother}
u_2 = S(u_1,A, b,\nu)
\end{equation}
a smoothing step performing a smoother (Jacobi or Gauss-Seidel) $\nu$ times on $u_1$ to obtain $u_2$. For the $i^{\rm th}$-level equation $A_h^i(\bmu)u_h^i(\bmu) = b(\bmu)$, one iteration of the V-cycle MG method at level $i$ is presented in Algorithm \ref{alg:MG}, denoted by $u(\bmu) = {\rm MG}(b(\bmu),A_h^i(\bmu),i)$. Using the MG method as a preconditioner of the PCG method, i.e., replacing the Step \ref{PCG_precondition} of Algorithm \ref{alg:PCG} by $s_{k+1}(\bmu) = {\rm MG}(r_h^{(k+1)}(\bmu),A_h(\bmu),J)$, we obtain the MGCG method. 
\[
    {\rm MGCG}(A_h(\bmu), f_h(\bmu), u_h^{(k-1)}(\bmu)) \coloneqq {\rm PCG}(A_h(\bmu), f_h(\bmu), u_h^{(k-1)}(\bmu); {\rm MG}(\cdot,A_h(\bmu),J)).
\]

\begin{algorithm}[htb]
\caption{MG V-cycle at level $i$: $u(\bmu) = {\rm MG}(b(\bmu),A_h^i(\bmu),i)$}
\label{alg:MG}
\begin{algorithmic}[1]
\STATE \textbf{Input:} $\{A_h^i(\bmu)\}_{i=0}^J$, $\{p_i(\bmu)\}_{i=0}^{J-1}$, $b(\bmu)$, the number of pre-smoothing steps $\nu$, and level $i$
\STATE Implement the pre-smoothing process $\widetilde{u}(\bmu) = S(0,A_h^i(\bmu),b(\bmu),\nu)$
\STATE Compute residual $r_i(\bmu) = b(\bmu)-A_h^i(\bmu)\widetilde{u}(\bmu)$
\STATE Restrict residual $r_{i-1}(\bmu) = p_{i-1}^T(\bmu)r_{i}(\bmu)$ 
\STATE Correct the error on the coarse grid:\\
\textbf{if} $i=1$, $e_0(\bmu) = (A_h^0(\bmu))^{-1}r_{0}(\bmu)$\\
\textbf{else} $e_{i-1}(\bmu) = {\rm MG}(r_{i-1}(\bmu), A_h^{i-1}(\bmu),i-1)$
\STATE Prolongate coarse grid correction $\overline{u}(\bmu) = \widetilde{u}(\bmu)+p_{i-1}(\bmu)e_{i-1}(\bmu)$
\STATE Implement the post-smoothing process $u(\bmu) = S(\overline{u}(\bmu),A_h^i(\bmu),b(\bmu),\nu)$
\end{algorithmic}
\end{algorithm}

\subsection{Multispace reduced basis preconditioner}\label{Sec: MSRB precondioner}
Multispace reduced basis (MSRB) preconditioner \cite{Dal2018MultiSpaceReducedBasis} combines a fine grid preconditioner, such as Jacobi and Gauss-Seidel preconditioner with a coarse preconditioner induced from a reduced basis solver. We consider the equation $A_h(\bmu)u_h(\bmu) = b(\bmu)$. With $n_s$ snapshots $\{u_h(\bmu_i)\}_{i=1}^{n_s}$ (that are the high-fidelity solutions computed for the training parameters $\Xi_{\rm train}=\{\bmu_i\}_{i=1}^{n_s}$), the proper orthogonal decomposition (POD) is used to build the RB space of dimension $N$, represented by the column space of a matrix $W_N = [w_1, \dots w_N]\in \mathbb{R}^{\N\times N}$. Then we can obtain an RB approximation $u_h^{rb}(\bmu) = W_Nu_N(\bmu)$ by solving the following reduced linear system
 \begin{equation}\label{RB system} 
 A_N(\bmu)u_N(\bmu) = b_N(\bmu).
 \end{equation}
 where $A_N(\bmu)\in \mathbb{R}^{N\times N}$ and $b_N(\bmu)\in \mathbb{R}^N$ are obtained by projecting $A_h(\bmu)$ and $b(\bmu)$ to the RB space
 \[
 A_N(\bmu) = W_N^TA_h(\bmu)W_N,\ b_N(\bmu) = W_N^Tb(\bmu).
 \]
 Thus the RB solution can be expressed as
 \[
 u_h^{rb}(\bmu) = W_NA_N^{-1}(\bmu)b_N(\bmu).
 \]
 Let the application of the above POD-based RBM with the RB dimension $N$ be denoted by 
 \begin{equation}\label{POD-based RBM}
    u_h^{rb}(\bmu) = {\rm RBM_{POD}}(A_h(\bmu),b(\bmu),N).
\end{equation}
Combining a fine smoother denoted as \eqref{smoother} with the RB preconditioner \eqref{POD-based RBM}, the MSRB preconditioner is presented in Algorithm \ref{alg:MSRB}, denoted by $u(\bmu) = {\rm MSRB}(b(\bmu),A_h(\bmu),N)$.  Replacing the Step \ref{PCG_precondition} of Algorithm \ref{alg:PCG} by $s_{k+1}(\bmu) = {\rm MSRB}(r_h^{(k+1)}(\bmu),A_h(\bmu),N^{(k)})$ gives the MSRBCG method:
\[
 {\rm MSRBCG}(A_h(\bmu), f_h(\bmu), u_h^{(k-1)}(\bmu)) \coloneqq {\rm PCG}(A_h(\bmu), f_h(\bmu), u_h^{(k-1)}(\bmu); {\rm MSRB}(\cdot,A_h(\bmu),N^{(k)})).
 \]
Here the $k^{\rm th}$ RB space $W_{N^{(k)}}$ of dimension $N^{(k)}$ needs to be specifically built based on the error snapshots $\{e_h^{(k)}(\bmu_i)\}_{i=1}^{n_s}$ that are the high-fidelity solutions of the $k^{\rm th}$ error equation $A_h(\bmu)e_h^{(k)}(\bmu) = r_h^{(k)}(\bmu)$. 

  \begin{algorithm}[htb]
\caption{MSRB preconditioner with $W_N$: $u(\bmu) = {\rm MSRB}(b(\bmu),A_h(\bmu),N)$}
\label{alg:MSRB}
\begin{algorithmic}[1]
\STATE \textbf{Input:} $A_h(\bmu)\in \mathbb{R}^{\N\times\N}$, $b(\bmu)\in \mathbb{R}^{\N}$, and the RB dimension $N$
\STATE Implement the fine-smoothing process $\widetilde{u}(\bmu) = S(0,A_h(\bmu),b(\bmu),1)$
\STATE Compute residual $r(\bmu) = b(\bmu)-A_h(\bmu)\widetilde{u}(\bmu)$ 
\STATE compute $u_h^{rb}(\bmu) = {\rm RBM_{POD}}(A_h(\bmu),r(\bmu),N)$
\STATE Update $u(\bmu) = \widetilde{u}(\bmu) + u_h^{rb}(\bmu)$
\end{algorithmic}
\end{algorithm}

\section{Reduced basis warm-start iterative solvers}
\label{RBM-initialized iterative solvers}
In this section, we present the reduced basis warm-start iterative method.  The method relies on the L1-based reduced over-collocation (L1ROC) method \cite{Chen2021EIM,Chen2021L1Based}, a highly efficient variant of the reduced basis method, to provide a controllable high-quality initial prediction. The PCG method presented in Algorithm \ref{alg:PCG} is then used for refining the RB initialization. To provide some theoretical and intuitive footing of the new method's superior performance over the MSRBCG algorithm, we provide a formal analysis of the MSRB preconditioner presented in Algorithm \ref{alg:MSRB}.

\subsection{L1-based reduced over-collocation method}
The L1ROC method is a greedy-based reduced basis method, where a greedy algorithm adaptively chooses snapshots by finding the parameter location at which the error estimate is maximum. This sequential sampling framework and its avoidance of truncation after POD-type pervasive sampling make the algorithm highly efficient and much less data-intensive.

Two key points further ensure the high efficiency of the L1ROC method. First is the direct adoption as the online solver of the discrete empirical interpolation method (DEIM) \cite{chaturantabut2010nonlinear, Barrault2004empirical}. Assume we have built a RB matrix $W_N = [w_1, \dots w_N]\in \mathbb{R}^{\N\times N}$, and a set of collocations points $X^M = \{x_1,\dots,x_M\}$ ($M = 2 N - 1$). We denote by $\chi = [\chi_1,\dots,\chi_M]$ the subset of the entries corresponding to the points $X^M$. For simplicity of notation, we introduce a sub-sampling matrix as 
\begin{equation*}
P = [\pmb e_{\chi_1},\dots,\pmb e_{\chi_M}]^T\in \bbR^{M \times \calN},
\end{equation*}
where $\pmb e_{\chi_i} = [0,\dots,0,1,0,\dots,0]^T\in \bbR^{\calN}$ denotes the unit vector whose $\chi_i$-th component equals $1$. Then the RB approximation $u_h^{rb}(\bmu) = W_Nu_N(\bmu)$ is given by solving the minimum square error estimate of the following sub-sampled system
\begin{equation}\label{RB system1}
PA_h(\bmu)W_Nu_N(\bmu) \approx Pb(\bmu),
\end{equation}
namely,
\[
u_N(\bmu)=\argmax_{c\in\mathbb{R}^N}\Vert P(b(\bmu)-A_h(\bmu)W_Nc)\Vert_{\mathbb{R}^M}.
\]
The other point is the proposal of an efficient error indicator that relies on the L1 norm of the RB coefficient with respect to the chosen snapshots \cite{chen2019robust}. Assume we have selected $n$ parameters $\{\bmu_{l_i}\}_{i=1}^n$ from 
the training parameter set $\Xi_{\rm train}=\{\bmu_i\}_{i=1}^{n_s}$ and obtained the snapshots $\{u_h(\bmu_{l_i})\}_{i=1}^{n_s}$ (that are unorthogonalized RB vectors). With the RB matrix $U_n = [u_h(\bmu_{l_1}),\dots,u_h(\bmu_{l_n})]$, the RB approximation $u_h^{rb}(\bmu)$ can be represented as $u_h^{rb}(\bmu) = U_nc_n(\bmu)$. The L1-based error indicator is presented as follows
\[
\Delta^{L1}_n(\bmu) = \Vert c_n(\bmu)\Vert_{l_{1}}.
\]
Based on this, the RB space $W_N$ and the collocation points $X^M$ are built by a greedy algorithm efficiently, which is presented in Algorithm \ref{alg:offline L1ROC}. Note that
\[
(u_n,x_{\star}^n)= {\rm DEIM}(W_{n-1}, u_h(\bmu_{l_n}))
\]
in Step \ref{offline_DEIM} denotes the application of the DEIM process on a new snapshot $u_h(\bmu_{l_n})$ with respect to the previous RB space $W_{n-1}$. The return result $u_n$ is the orthonormal (under point evaluation) RB vector and $x_{\star}^n$ is the corresponding interpolation point. We denote the application of the L1ROC method including the offline and online process by 
\begin{equation}\label{L1ROC}
    u_h^{rb}(\bmu) = {\rm RBM_{L1ROC}}(A_h(\bmu),b(\bmu),N).
\end{equation}
\begin{algorithm}[htb]
\caption{Offline algorithm for L1ROC}
\label{alg:offline L1ROC}
\begin{algorithmic}[1]
\STATE \textbf{Input:} The training parameter set $\Xi_{\rm train}$ and the dimension of the RB space $N$
\STATE Initialize $W_{0} = R_0=X_s^0=X_r^0=\emptyset$ 
\STATE Choose $\bmu_{l_1}$ randomly in $\Xi_{\rm train}$ and obtain $u_h(\bmu_{l_1})$ by Algorithm \ref{alg:PCG}. Find $(u_1,x_{\star}^1)= {\rm DEIM}(W_{0}, u_h(\bmu_{l_1}))$. Then let $n = 1$, $m=1$, $X_{s}^{n} = \{x_{\star}^{1}\}$, $X^{m} = X_{s}^{n}\cup X_{r}^{n-1}$ and $W_{1} = [u_{1}]$.
\FOR{$n = 2,\cdots, N$}
\STATE Solve $u_{n-1}(\bmu)$ by the system \eqref{RB system1} with $W_{n-1}, X^{m}$ and calculate $\Delta_{n-1}^{L1}$ for every $\bmu\in \Xi_{\rm train}$.\label{alg1_RB_approximation}
\STATE Find $\bmu_{l_n} = \mathop{\arg\max}_{\bmu\in \Xi_{\rm train}}\Delta_{n-1}^{L1}(\bmu)$.
\STATE Solve $u_h(\bmu_{l_n})$ by Algorithm \ref{alg:PCG}. Find $(u_n,x_{\star}^n)= {\rm DEIM}(W_{n-1}, u_h(\bmu_{l_n}))$ and let $X_{s}^{n} = X_{s}^{n-1}\cup\{x_{\star}^{n}\}$.\label{offline_DEIM}
\STATE Compute $r_{n-1}(\bmu_{l_n}) = b(\bmu_{l_n})-A_h(\bmu_{l_n})W_{n-1}u_{n-1}(\bmu_{l_n})$. Find $(r_{n-1},x_{\star\star}^{n-1})= {\rm DEIM}(R_{n-2}, r_{n-1}(\bmu_{l_n}))$ and let $X_{r}^{n-1} = X_{r}^{n-2}\cup\{x_{\star\star}^{n-1}\}$.
\STATE Update $W_n = [W_{n-1},u_n]$, $m = m+2$, $X^m = X_{s}^{n}\cup X_{r}^{n-1}$.
\ENDFOR
\end{algorithmic}
\end{algorithm}

\subsection{Main algorithm}
Inspired by the neural network warm-start approaches \cite{zhou2023neural,nikolopoulosa2022ai} and afforded by the highly efficient online solver of the L1ROC method, we introduce the Reduced Basis Warm-Start (RBWS) iterative method. It features 
an offline training process which adds minimum overhead cost thanks to the adoption of the highly efficient L1ROC method. After this training stage, RBWS employs the L1ROC online solver to generate an accurate RB solution for each new system as the initial value. Then the PCG method presented in Algorithm \ref{alg:PCG} is applied to refine the initial value towards the exact solution. 

This RBM initialized PCG (RBI-PCG) method is presented in Algorithm \ref{alg:RBMI PCG}, where either \eqref{POD-based RBM} or \eqref{L1ROC} can be adopted as ${\rm RBM}(A_h(\bmu),f_h(\bmu),N)$. Depending on what specific ${\mathcal P}(\cdot)$ the algorithm takes, it leads to the RBM initialized MGCG (RBI-MGCG) method or the RBM initialized MSRBCG (RBI-MSRBCG) method when replacing the ${\mathcal P}(\cdot)$ in Step \ref{RBMI-PCG iteration} by the MG method of Algorithm \ref{alg:MG} or the MSRB of Algorithm \ref{alg:MSRB}.
\begin{algorithm}[htb]
\caption{RBI-PCG algorithm, generating RBI-MGCG and RBI-MSRBCG}
\label{alg:RBMI PCG}
\begin{algorithmic}[1]
\STATE \textbf{Input:} $A_h(\bmu)\in \mathbb{R}^{\N\times\N}$, $f_h(\bmu)\in \mathbb{R}^{\N}$, the RB dimension $N$, the residual tolerance $\delta$, and the maximum number of iterations $L_{\rm max}$
\STATE Generate an initial value by solving the RB system: $u_h^{(0)}(\bmu)= {\rm RBM}(A_h(\bmu),f_h(\bmu),N)$
\STATE Compute initial residual $r_h^{(0)}(\bmu) = f_h(\bmu)-A_h(\bmu)u_h^{(0)}(\bmu)$ and set $k=0$ 
\WHILE{$\Vert r_h^{(k)}(\bmu)\Vert/\Vert f_h(\bmu)\Vert<\delta$ $\&$ $k< L_{\rm max}$}
\STATE $u_h^{(k)}(\bmu) = {\rm PCG}(A_h(\bmu), f_h(\bmu), u_h^{(k-1)}(\bmu); {\mathcal P}(\cdot))$\label{RBMI-PCG iteration}
\STATE $r_h^{(k)}(\bmu) = f_h(\bmu)-A_h(\bmu)u_h^{(k)}(\bmu)$
\STATE $k = k+1$
\ENDWHILE
\end{algorithmic}
\end{algorithm}

\subsubsection{MSRB Update: Factors toward its low efficiency.}

\label{Sec: RBM preconditioner}

In this section, we aim to provide some insight into the performance of the MSRB preconditioner described in Section \ref{Sec: MSRB precondioner} in the RBM-initialized iterative method for the high-precision solution.
 For simplicity, we consider the MSRB preconditioned Richardson method that is easily rewritten as
\begin{equation}\label{iterative RBM}
\begin{cases}
u_h^{(0)}(\bmu) = {\rm RBM}(A_h(\bmu),f_h(\bmu),N),\\
u_h^{(k-\frac{1}{2})}(\bmu) = u_h^{(k-1)}(\bmu) +S(0,A_h(\bmu),r_h^{(k-1)}(\bmu),1),\ k = 1,2,\dots,\\
u_h^{(k)}(\bmu) = u_h^{(k-\frac{1}{2})}(\bmu) +{\rm RBM}(A_h(\bmu),r_h^{(k-\frac{1}{2})}(\bmu),N^{(k)}),\ k = 1,2,\dots.\\
\end{cases}
\end{equation}
Here $r_h^{(l)}(\bmu) = f_h(\bmu)-A_h(\bmu)u_h^{(l)}(\bmu)$ denotes the residual corresponding to the iterative solution $u_h^{(l)}(\bmu)$. 
The method relies on corrections afforded by the resolution of the error equations where we denote the error of the $k^{\rm th}$ iterative solution by $e_h^{(k)}(\bmu) = u_h(\bmu)-u_h^{(k)}(\bmu)$
\begin{equation}
\label{eq:erroreqn}
\begin{cases}
A_h(\bmu)e_h^{(k-1)}(\bmu)= r_h^{(k-1)}(\bmu)  & \mbox{(by the smoother)} \\
A_h(\bmu)e_h^{(k-\frac{1}{2})}(\bmu)= r_h^{(k-\frac{1}{2})}(\bmu) &   \mbox{(by the RBM online solver)}.
\end{cases}
\end{equation}
The first one is a full order model that is, albeit expensive, well-known (e.g. Multigrid literature) to be effective in driving the approximation to convergence. However, the second one, less well-understood, relies on {\em the low-rank approximability of the error manifold}. This strategy comes with two challenges.

\noindent {\bf 1. Improving accuracy requirement on RBM as $k\uparrow$.} Because this is a correction step for $u_h^{(k-\frac{1}{2})}(\bmu)$, its accuracy should be above the error committed by $u_h^{(k-\frac{1}{2})}(\bmu)$. As shown by the next lemma, this means that the RB dimension will have to increase with respect to the iteration index $k$.
\begin{lemma}
    Given that the linear system is well-conditioned and the Kolmogorov n-width of the error manifold
    \[
    \calW \coloneqq \{A_h(\bmu)^{-1}(r_h^{(k-\frac{1}{2})}(\bmu)): \bmu \in \calD\}
    \]
    decreases at a polynomial or exponential rate uniformly with respect to $k$, it follows that the dimension $n$ of the reduced basis manifold must increase as the iteration index $k$ increases. 
    \end{lemma}
\begin{proof}
It is easy to see, from \eqref{iterative RBM}, that
\[
e_h^{(k)}(\bmu) = e_h^{(k-\frac{1}{2})}(\bmu) - \widetilde{e}_h^{(k-\frac{1}{2})}(\bmu)
\]
where $\widetilde{e}_h^{(k-\frac{1}{2})}(\bmu)$ represents the ($n$-dimensional) RBM online approximation of $e_h^{(k-\frac{1}{2})}(\bmu)$ by \eqref{eq:erroreqn}. This is lower-bounded by the best approximation error 
\[
\lVert e_h^{(k-\frac{1}{2})}(\bmu) - P_{W_n} e_h^{(k-\frac{1}{2})}(\bmu) \rVert \le \sigma_n(\calW) \coloneqq \max_{w \in \calW} \lVert w - P_{W_n} w \rVert
\]
where $P_{W_n}$ denotes the projection into the RB space $W_n$. We have that, with the RBM greedy algorithm, $\sigma_n(\calW)$ inherits the same rate of decay of the Kolmogorov $n$-width $d_n(\calW)$ as follows\cite{Binev2011Convergence}\footnote{To simplify our formal analysis, we assume that the upper bounds on the decay rates in \cite{Binev2011Convergence} are actually attainable, as typically confirmed numerically in the RBM literature.} :
\begin{itemize}
    \item {\bf Polynomial decay.} If $d_n(\calW) \approx M n^{-\alpha}$, then $\sigma_n(\calW) \approx C M n^{-\alpha}$ with $C\coloneqq q^{\frac{1}{2}}(4 q)^\alpha$ and $q\coloneqq \lceil 2^{\alpha+1} \gamma^{-1}\rceil^2$. Here $M \coloneqq \max_{w\in\calW} \lVert w \rVert$ and $\gamma$ is the norm-equivalency constant between $\lVert \cdot \rVert$ and the error estimator adopted by the RBM. Given that $A_h(\bmu)$ is well-conditioned, we have $M = \gamma^{-1} \left(\max_{\bmu \in \calD} \lVert r_h^{(k-\frac{1}{2})}(\bmu) \rVert \right)$ and therefore 
    \[
    \lVert e_h^{(k)}(\bmu) \rVert \approx C \gamma^{-1} \left(\max_{\bmu \in \calD} \lVert r_h^{(k-\frac{1}{2})}(\bmu) \rVert \right) n^{-\alpha}.
    \]
    Given that $C$ is dependent on $\alpha$ (which is assumed to be uniform with $k$) and $\gamma$ (which is determined solely by $A_h(\bmu)$), we conclude that for $\lVert e_h^{(k)}(\bmu) \rVert$ to decay with respect to $k$, the RB dimension $n$ should increase as $k$ increases (i.e. the iteration goes on).
    \item {\bf Exponential decay.} The result follows similarly as in the case of polynomial decay thanks to the inheritance of the decay rate of $\sigma_n(\calW)$ proved in \cite{Binev2011Convergence}.
\end{itemize}
\end{proof}
\noindent {\bf 2. Degradation of the low-rank structure.} We first define the residual manifold at step $k$
\[
 \calR^{(k)} \coloneqq \left \{ r_h^{(k-\frac{1}{2})}(\bmu) = f_h(\bmu)-A_h(\bmu)u_h^{(k-\frac{1}{2})}(\bmu): \bmu \in \calD \right \}.
\]
As $k$ increases and $u_h^{(k-\frac{1}{2})}(\bmu)$ gets more accurate, $\lVert r_h^{(k-\frac{1}{2})}(\bmu) \rVert$ decreases. Since we aim to have $\lVert r_h^{(k-\frac{1}{2})}(\bmu) \rVert$ at the level of machine accuracy at convergence, elements of $r_h^{(k-\frac{1}{2})}(\bmu)$ will become more and more comparable to the round off error. This means that the Kolmogorov $n$-width
\[
d_n^k \coloneqq d_n(\calR^{(k)})
\]
will likely decay slower as $k$ increases. While this is confirmed by our numerical results (see Figure \ref{fig:widthchange} which also shows the decay rates for the residual manifolds), we intend to leave the theoretical proof of this degradation to future work.

\begin{remark}
The compounding impact of the two challenges enunciated above is that the RB space dimension must increase significantly as the iteration proceeds if we were to maintain the convergence rate of the iterative solver. However, this comes with a significant cost (see Appendix \ref{appendix:offline cost}) making the preconditioning not cost-effective. On the other hand, if we aim to control the computational cost (by e.g. fixing the RB dimension), the convergence rate will deteriorate as the iterative solver proceeds. This is confirmed by our numerical results next - the convergence of the RBWS does not deteriorate while the RBI-MSRBCG does.
\end{remark}

\section{Numerical tests}\label{sec:numerical result}
To test our algorithms, we set $\Omega = (0,1)^3$ with non-intersecting boundaries $\Gamma_D$ and $\Gamma_N$ such that $\partial \Omega = \Gamma_D\cup\Gamma_N$.  We consider the following parametrized diffusion equation
\begin{equation}\label{anisotropic diffusion eq}
	\left\{
	\begin{aligned}
		&-\nabla\cdot(\mathcal{K}(\bmu)\nabla u(\bmu)) = f(\bmu),\ {\rm in}\  \Omega,\\
&u(\bmu) = g_D(\bmu),\ {\rm on}\ \Gamma_D,\\
&\frac{\partial u(\bmu)}{\partial n} = g_N(\bmu),\ {\rm on}\ \Gamma_N.
	\end{aligned}
	\right.
\end{equation}
Here, the diffusion tensor $\mathcal{K}(\bmu)$, the source term $f(\bmu)$, and the boundary conditions $g_D(\bmu)$, $g_N(\bmu)$ may depend on parameter $\bmu\in \mathcal{D}$. Specifically, the following two examples are given as considered in \cite{nguyen2018reduced, Dal2018MultiSpaceReducedBasis} respectively.

\paragraph{\textbf{Example 1}} Dirichlet boundary value problem where the specific definitions are given by
\[
\mathcal{K}(\bmu) = 1+\mu_1(\sin(20\pi(4(x-\frac{1}{2})^2+(y-\frac{1}{2})^2+(z-\frac{1}{2})^2))^2,
\]
\[
f = 3\pi^2\sin(\pi x)\sin(\pi y)\sin(\pi z)
\]
\[
g_D(\bmu) = (1-\mu_2)\cos(10\pi(4(x-\frac{1}{2})^2+(y-\frac{1}{2})^2+(z-\frac{1}{2})^2))+\mu_2\cos(10\pi(x+y+z)).
\]
The parameter domain is $\mathcal{D}= [0,2]\times[0,1]$.

\paragraph{\textbf{Example 2}} Mixed boundary value problem
\[
\Gamma_N = \{\pmb x = (x,y,z)\in\overline\Omega:x = 1\},\ \Gamma_D = \partial\Omega\backslash\Gamma_N.
\]
The diffusion tensor is
\[
\mathcal{K}(\bmu) = \mathcal{K}(\pmb{x}; \bmu) = \nu(\pmb{x};\bmu){\rm diag}(1,1,10^{-2}),
\]
where $\nu(\pmb{x};\bmu)$ is the piecewise constant on four subregions $\Omega_1 = (0,1)\times(0,0.5)\times(0,0.5)$, $\Omega_2 = (0,1)\times(0,0.5)\times(0.5,1)$, $\Omega_3 = (0,1)\times(0.5,1)\times(0,0.5)$, and $\Omega_4 = (0,1)\times(0.5,1)\times(0.5,1)$, denoted by
\begin{equation*}
\nu(\pmb x;\bmu) = 
\begin{cases}
\mu_j,\ \pmb x \in \Omega_j,\ j = 1,\dots,3,\\
1,\ \ \pmb x \in \Omega_4.
\end{cases}
\end{equation*}
We consider the following parameter-dependent Gaussian function as the source term
\[
f(\pmb x;\bmu) = \mu_7 + \frac{\exp\left(-((x-\mu_4)^2+(y-\mu_5)^2+(z-\mu_6)^2)/\mu_7\right)}{\mu_7},
\]
and homogeneous boundary conditions
\[
g_D(\bmu) = 0,\ g_N(\bmu) = 0.
\]
The 7-dimensional parameter domain is $\mathcal{D}= [0.1,1]^3\times[0.4,0.6]^3\times[0.25,0.5]$.

\subsection{Results on convergence and efficiency of RBWS}

For both examples, linear finite elements as implemented in the Matlab package iFEM \cite{chen2009integrated} with DoFs $\N = 35,937$ are adopted as the high-fidelity discretization. Subsequently, we employ the three methods detailed in Table \ref{tab:methods} to solve the resulting discrete system \eqref{discrete equation}.
\begin{table}[thbp]
\renewcommand{\arraystretch}{1.1}
    \centering
    \begin{tabular}{|c|c|c|}
    \hline
    Method & Iterative solver & Initial guess\\
    \hline
    \textbf{MGCG} &  MGCG (Section \ref{Sec:MGCG}) & $u_h^{(0)}(\bmu)= 0$\\
    \textbf{RBI-MGCG} & MGCG (Section \ref{Sec:MGCG}) & $u_h^{(0)}(\bmu)= {\rm RBM_{L1ROC}}(A_h(\bmu),f_h(\bmu),N)$\\
    \textbf{RBI-MSRBCG} & MSRBCG (Section \ref{Sec: MSRB precondioner}) & $u_h^{(0)}(\bmu)= {\rm RBM_{POD}}(A_h(\bmu),f_h(\bmu),N)$\\
    \hline
    \end{tabular}
    \caption{Three methods tested in this paper.}
    \label{tab:methods}
\end{table}
For the MGCG method and the RBI-MGCG method, we use an MG method with $4$ levels ($J = 3$) as the preconditioner. The L1ROC method is used for the initialization of the RBI-MGCG method. The POD-based RBM is used for the initialization of the RBI-MSRBCG method. Here we consider a fixed RB dimension $N$ for the MSRB preconditioners at each iteration step, i.e., $N^{(k)}=N$.

For Example 1, 70 parameters are sampled by the popular Latin hypercube sampling (LHS) method \cite{Mckay2000LHS} to build the training set $\Xi_{\rm train}$ for the L1ROC method and POD-based RBM. Then we respectively construct the RB spaces of different dimensions $N = 10, 15, 20$ and test the corresponding RBM initialized iterative methods. The parametrized linear system is solved for a testing set $\Xi_{\rm test}$ consisting of $500$ parameters with the residual tolerance $\delta = 10^{-16}$ and the maximum number of iterations $L_{\rm max} = 40$. To eliminate the parametric variations, we calculate the average value of the norms of all relative residuals
\[
r_{\rm ave}^{(k)} = \frac{\sum_{\bmu\in \Xi_{\rm test}}\Vert r_h^{(k)}(\bmu)\Vert/\Vert f_h(\bmu)\Vert}{\# \Xi_{\rm test}},
\]
where $r_h^{(k)}(\bmu) = f_h(\bmu)-A_h(\bmu)u_h^{(k)}(\bmu)$. And the convergence results of the average residual $r_{\rm ave}^{(k)}$ as a function of the iteration $k$ are presented in Figure \ref{fig:convergence and efficiency results} top left. We can see that the RBM-initialized methods starting from more accurate initial values require fewer iterations than the method without such a warm start. The MSRB preconditioner could provide a significant acceleration for the CG method within a certain precision. However, when the accuracy increases further, the MSRB preconditioner is significantly degraded, which leads to a much slower convergence. This is consistent with what our formal analysis in Section \ref{Sec: RBM preconditioner} predicts.

To demonstrate the efficiency benefit brought by the RBWS method, we record the cumulative runtime as the number of linear solves increases, which is shown in Figure \ref{fig:convergence and efficiency results} top right. The values corresponding to zero solves represent the computational cost of the offline training process. It can be seen that the RBI-MGCG methods begin to pay off quickly when the parametrized system is solved about 60 times thanks to its high online-efficiency. For the RBI-MSRB methods, the construction of the MSRB preconditioners is much more time-consuming and generates much less marginal savings online compared with the RBI-MGCG method. 
For a more detailed comparison, we report in Tabel \ref{tab:Example1} the break-even point (BEP) for the two RBWS methods that is defined as
\[
{\rm BEP} = \frac{t_{\rm off}}{t_{\rm on}({\rm MGCG})-t_{\rm on}({\rm RBWS})}.
\]
Here, $t_{\rm off}$ denotes the computation time for the offline stage, while $t_{\rm on}$ means the average computation time for the online stage. We specify two different values for the tolerance $\delta$ and record the iteration step number $L$ required for the average residual $r_{\rm ave}^{(k)}$ to fall below $\delta$. We see that the advantage of RBM preconditioning (adopted by RBI-MSRBCG) disappears when we move from low ($\delta = 10^{-8}$) to high ($\delta = 10^{-16}$) precision.

For Example 2, we build the RB spaces of dimensions $N = 100,200,300$ based on 500 training parameters for the L1ROC method and POD-based RBM respectively, and test all methods for 30,000 testing parameters. The results of the convergence and the cumulative time are presented in the bottom row of Figure \ref{fig:convergence and efficiency results}. And the the comparison results of $L$, $t_{\rm off}$, $t_{\rm on}$, and BEP are shown in Table \ref{tab:Example2}. 
The comparison of RBI-MGCG and RBI-MSRBCG is consistent with Example 1.
\begin{figure}[htbp]
\centering
\subfigure{
\includegraphics[width=0.45\textwidth]{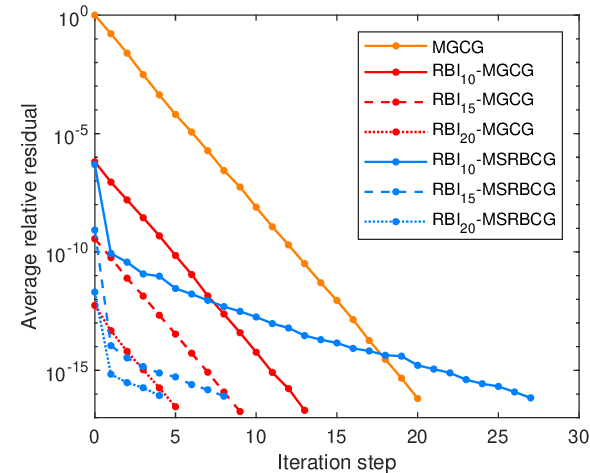}
}
\subfigure{
\includegraphics[width=0.45\textwidth]{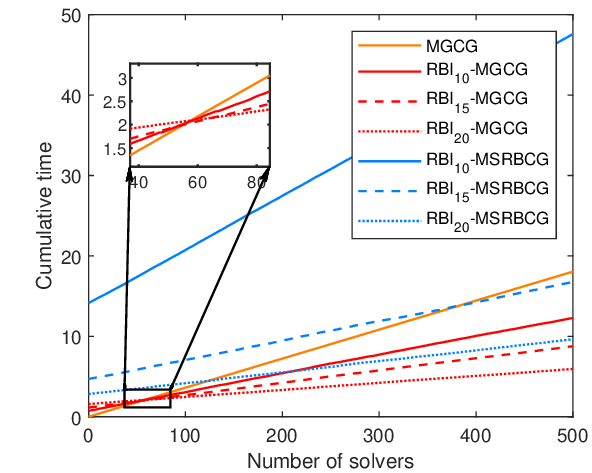}
}\\
\subfigure{
\includegraphics[width=0.45\textwidth]{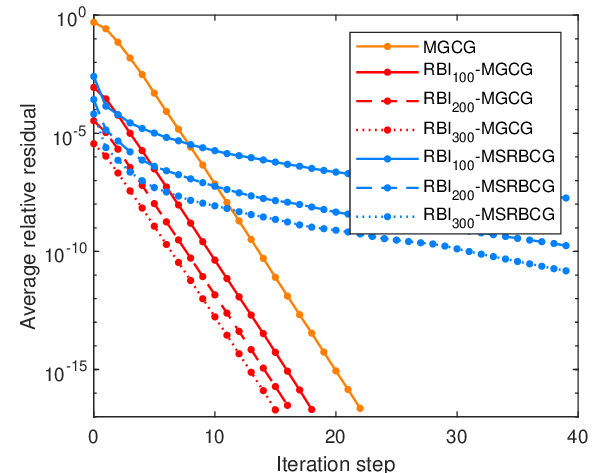}
}
\subfigure{
\includegraphics[width=0.45\textwidth]{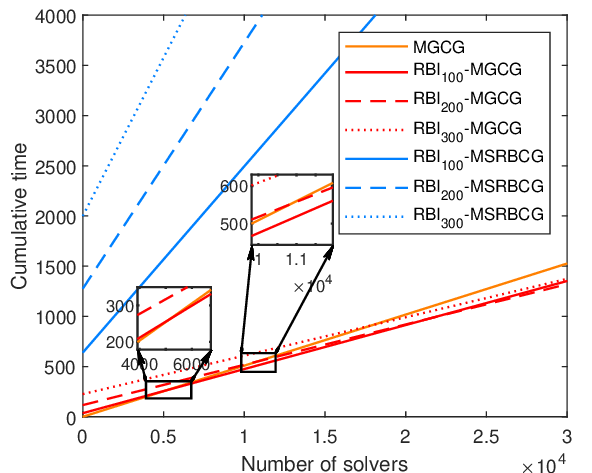}
}
\caption{Left: The convergence result with the iteration. Right: The total computational cost as the number of online solvers increases. (Top: Example 1, Bottom: Example 2.)}
\label{fig:convergence and efficiency results}
\end{figure}
\begin{table}[htbp]
\centering
{\fontsize{9.5}{14}\selectfont
\begin{tabular}{ | c | c|    c | c|c|c|}
\Xhline{1pt}
\multicolumn{1}{  |c |}{\multirow{2}{*}{Method}}
&
\multicolumn{1}{  c |}{\multirow{2}{*}{N}}
&\multicolumn{4}{ c |}{$\delta = 10^{-8}$}\\ \cline{3-6}
&&$L$&$t_{\rm off}$ &$t_{\rm on}$&BEP\\\hline
\textbf{MGCG}& --&  10&   -- &  1.87E-02 &--  \\
\hline
\textbf{RBI-}& 10&  3&7.31E-01&   7.10E-03 &   61   \\
\textbf{MGCG}& 15 &0 &1.15E00&   8.26E-04  &63    \\
& 20 &0&1.58E00&   9.32E-04  &86     \\
\hline
\textbf{RBI-}& 10&1  &6.86E-01&   4.62E-03 &   49   \\
\textbf{MSRBCG}& 15 &0&8.92E-02&   1.54E-03  &7     \\
& 20 &0&9.51E-02&   1.72E-03  &7     \\
\Xhline{1pt}
\end{tabular}
\begin{tabular}{ |c| c| c|c|}
\Xhline{1pt}
\multicolumn{4}{ |c |}{$\delta = 10^{-16}$}	\\\cline{1-4}
$L$&$t_{\rm off}$ &$t_{\rm on}$&BEP	\\\hline
20& --    &  3.61E-02 & -- \\
\hline
13& 7.31E-01&2.30E-02 &58  \\
9& 1.15E00&1.52E-02 &56   \\
5& 1.58E00&8.79E-03 &59   \\
\hline
27& 1.41E01&6.63E-02 &$\infty$   \\
8& 4.70E00&2.39E-02 &383  \\
4& 2.82E00&1.38E-02 &125   \\
\Xhline{1pt}
\end{tabular}
}
\caption{Detailed results with different $\delta = 10^{-8}, 10^{-16}$ for Example 1}
\label{tab:Example1}
\end{table}

\begin{table}[htbp]
\centering
{\fontsize{9.5}{14}\selectfont
\begin{tabular}{ | c | c|    c | c|c|c|}
\Xhline{1pt}
\multicolumn{1}{  |c |}{\multirow{2}{*}{Method}}
&
\multicolumn{1}{  c |}{\multirow{2}{*}{N}}
&\multicolumn{4}{ c |}{$\delta = 10^{-8}$}\\ \cline{3-6}
&&$L$&$t_{\rm off}$ &$t_{\rm on}$&BEP\\\hline
\textbf{MGCG}& --&  11&   -- &  1.49E-02 &--  \\
\hline
\textbf{RBI-}& 100&  7&3.56E01&   1.04E-02 &   3929   \\
\textbf{MGCG}& 200 &6 &1.16E02&   8.63E-03  &9248    \\
& 300 &4&2.26E02&   7.58E-03  &15374     \\
\hline
\textbf{RBI-}& 100&$\ge40$  &$\ge$6.39E02&   $\ge$9.45E-02 &   --   \\
\textbf{MSRBCG}& 200 &17&5.45E02&   6.57E-02 &   $\infty$     \\
& 300 &10&4.66E02&   5.19E-02 &   $\infty$    \\
\Xhline{1pt}
\end{tabular}
\begin{tabular}{ |c| c| c|c|}
\Xhline{1pt}
\multicolumn{4}{ |c |}{$\delta = 10^{-16}$}	\\\cline{1-4}
$L$&$t_{\rm off}$ &$t_{\rm on}$&BEP	\\\hline
22& --    &  2.67E-02 & -- \\
\hline
18& 3.56E01&2.20E-02 &5132   \\
16& 1.16E02&2.01E-02 &10829   \\
15& 2.26E02&1.91E-02 &17910   \\
\hline
$\ge40$& $\ge$6.39E02&$\ge$9.28E-02 &--    \\
$\ge40$& $\ge$1.28E03&$\ge$1.23E-02 &--   \\
$\ge40$& $\ge$2.00E03&$\ge$1.58E-02 &--    \\
\Xhline{1pt}
\end{tabular}
}
\caption{Detailed results with different $\delta = 10^{-8}, 10^{-16}$ for Example 2}
\label{tab:Example2}
\end{table}

\subsection{Numerical confirmation of the efficiency degradation of the RB preconditioner}

To confirm our analysis in Section \ref{Sec: RBM preconditioner}, we estimate the approximation accuracy of the RB space by computing the maximum relative residual of the RB solutions solved by \eqref{L1ROC} over the test set,
\[
r_N = \max_{\bmu\in \Xi_{\rm test}}\frac{\Vert f_h(\bmu)-A_h(\bmu)u_h^{rb}(\bmu)\Vert}{\Vert f_h(\bmu)\Vert}.
\]
The convergence of $r_N$ with the RB dimension $N$ is presented in Figure \ref{fig:offline cost}. Comparing the approximation accuracy of the RB space with the convergence results of the RBI-MSRBCG methods in Figure \ref{fig:convergence and efficiency results}, we can conclude that the preconditioner starts to be less efficient shortly after the iterative accuracy exceeds the approximation accuracy of the RB space. Moreover, we check the decay rate of the Kolmogorov $n$-width of the residual manifold $d_h^k= d_n(\calR^{(k)})$ by running a POD on the residual snapshot collection 
\[
 \overline{\calR}^{(k)} \coloneqq \left \{ r_h^{(k-\frac{1}{2})}(\bmu) = f_h(\bmu)-A_h(\bmu)u_h^{(k-\frac{1}{2})}(\bmu): \bmu \in \Xi_{\rm train}=\{\bmu_n\}_{n=1}^{n_s} \right \},
\]
which are obtained by the RBI-MSRBCG method with the RB dimension $N=20$ for Example 1 and $N = 300$ for Example 2. The rate of decay of the relative eigenvalues $\lambda_n/\lambda_{\rm max}$ is demonstrated in Figure \ref{fig:widthchange}. It's clear that the relative eigenvalues decay fast during the first two iterations, and decrease much slower as the iteration goes on.
\begin{figure}[htbp]
\centering
\subfigure{
\includegraphics[width=0.45\textwidth]{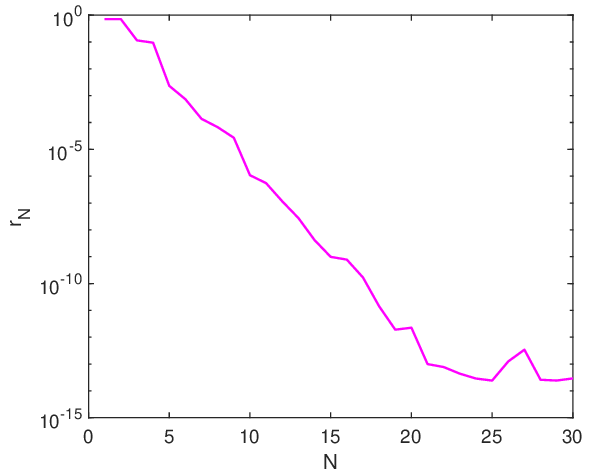}
}
\subfigure{
\includegraphics[width=0.45\textwidth]{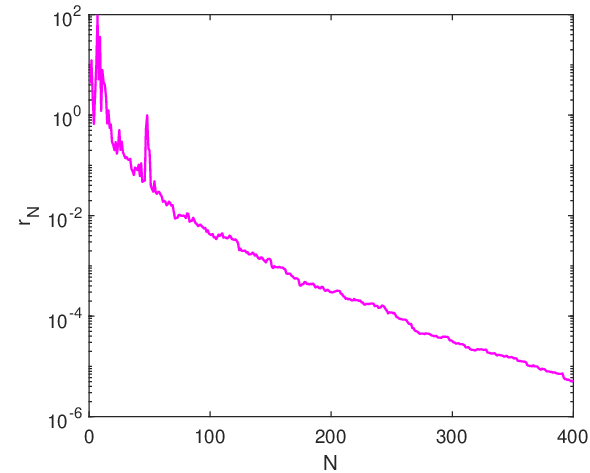}
}
\caption{The convergence of the relative residual $r_N$ with the RB dimension $N$. (Left: Example 1, Right: Example 2.)}
\label{fig:offline cost}
\end{figure}
\begin{figure}[htbp]
\centering
\subfigure{
\includegraphics[width=0.45\textwidth]{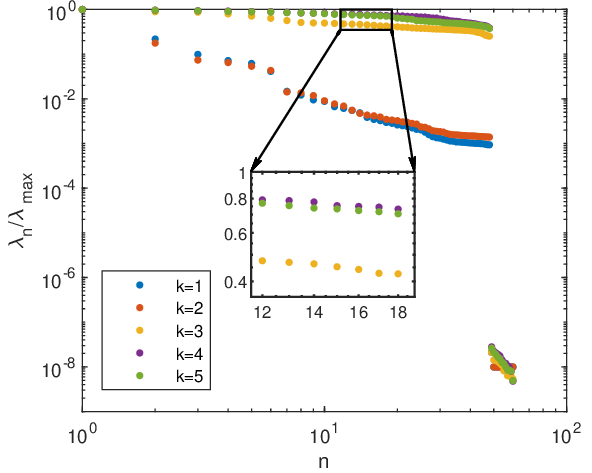}
}
\subfigure{
\includegraphics[width=0.45\textwidth]{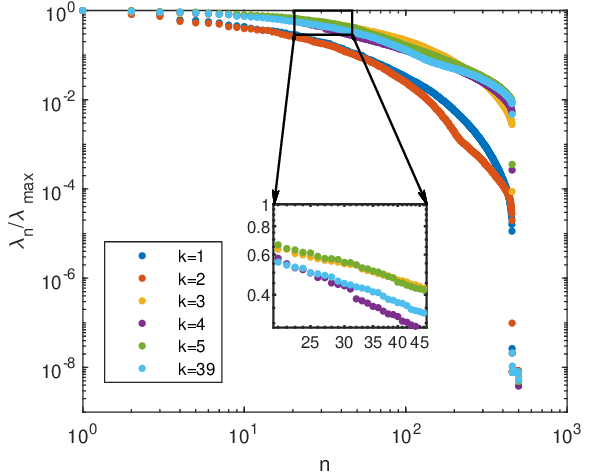}
}
\caption{Decay of the eigenvalues for the residual snapshots computed by the RBI-MSRBCG method. (Left: Example 1, Right: Example 2.)}
\label{fig:widthchange}
\end{figure}

\section{Conclusion}\label{sec:Conclusion}
We propose and test a reduced basis warm-start approach employing a highly efficient RBM variant to initialize the high-fidelity iterative method to obtain the high-precision solutions of the parametrized linear systems. Moreover, we discuss the efficiency limitation of RBM, for situations when solutions with machine accuracy are sought, and when it is adopted as a preconditioner in the iterative methods. The numerical results demonstrate the advantage of the RBWS initialization and verify the limitation of RBM as a preconditioner.

\bibliographystyle{abbrv}
\bibliography{ref}

\begin{thebibliography}{10}

\bibitem{Barrault2004empirical}
M.~Barrault, Y.~Maday, N.~C. Nguyen, and A.~T. Patera.
\newblock An `empirical interpolation' method: application to efficient
  reduced-basis discretization of partial differential equations.
\newblock {\em C. R. Math. Acad. Sci. Paris}, 339(9):667--672, 2004.

\bibitem{Binev2011Convergence}
P.~Binev, A.~Cohen, W.~Dahmen, R.~DeVore, G.~Petrova, and P.~Wojtaszczyk.
\newblock Convergence rates for greedy algorithms in reduced basis methods.
\newblock {\em SIAM J. Math. Anal.}, 43(3):1457--1472, 2011.

\bibitem{brandt1977multi}
A.~Brandt.
\newblock Multi-level adaptive solutions to boundary-value problems.
\newblock {\em Math. Comp.}, 31(138):333--390, 1977.

\bibitem{Buffa2012prioriconvergence}
A.~Buffa, Y.~Maday, A.~T. Patera, C.~Prud'homme, and G.~Turinici.
\newblock {\it {A} priori} convergence of the greedy algorithm for the
  parametrized reduced basis method.
\newblock {\em ESAIM Math. Model. Numer. Anal.}, 46(3):595--603, 2012.

\bibitem{chaturantabut2010nonlinear}
S.~Chaturantabut and D.~C. Sorensen.
\newblock Nonlinear model reduction via discrete empirical interpolation.
\newblock {\em SIAM J. Sci. Comput.}, 32(5):2737--2764, 2010.

\bibitem{chen2009integrated}
L.~Chen.
\newblock An integrated finite element method package in matlab.
\newblock {\em California: University of California at Irvine}, 2009.

\bibitem{Chen2021EIM}
Y.~Chen, S.~Gottlieb, L.~Ji, and Y.~Maday.
\newblock An {EIM}-degradation free reduced basis method via over collocation
  and residual hyper reduction-based error estimation.
\newblock {\em J. Comput. Phys.}, 444:Paper No. 110545, 18, 2021.

\bibitem{Chen2021L1Based}
Y.~Chen, L.~Ji, A.~Narayan, and Z.~Xu.
\newblock L1-based reduced over collocation and hyper reduction for steady
  state and time-dependent nonlinear equations.
\newblock {\em J. Sci. Comput.}, 87(1):Paper No. 10, 21, 2021.

\bibitem{chen2019robust}
Y.~Chen, J.~Jiang, and A.~Narayan.
\newblock A robust error estimator and a residual-free error indicator for
  reduced basis methods.
\newblock {\em Comput. Math. Appl.}, 77(7):1963--1979, 2019.

\bibitem{Cohen2020Reduced}
A.~Cohen, W.~Dahmen, R.~DeVore, and J.~Nichols.
\newblock Reduced basis greedy selection using random training sets.
\newblock {\em ESAIM Math. Model. Numer. Anal.}, 54(5):1509--1524, 2020.

\bibitem{Dal2018MultiSpaceReducedBasis}
N.~Dal~Santo, S.~Deparis, A.~Manzoni, and A.~Quarteroni.
\newblock Multi space reduced basis preconditioners for large-scale
  parametrized {PDE}s.
\newblock {\em SIAM J. Sci. Comput.}, 40(2):A954--A983, 2018.

\bibitem{Dal2018Stokesequations}
N.~Dal~Santo, S.~Deparis, A.~Manzoni, and A.~Quarteroni.
\newblock Multi space reduced basis preconditioners for parametrized {S}tokes
  equations.
\newblock {\em Comput. Math. Appl.}, 77(6):1583--1604, 2019.

\bibitem{DeVore2012Greedy}
R.~DeVore, G.~Petrova, and P.~Wojtaszczyk.
\newblock Greedy algorithms for reduced bases in {B}anach spaces.
\newblock {\em Constr. Approx.}, 37(3):455--466, 2013.

\bibitem{Haasdonk2017Review}
B.~Haasdonk.
\newblock Reduced basis methods for parametrized {PDE}s---a tutorial
  introduction for stationary and instationary problems.
\newblock In {\em Model reduction and approximation}, volume~15 of {\em Comput.
  Sci. Eng.}, pages 65--136. SIAM, Philadelphia, PA, 2017.

\bibitem{hackbusch2013multi}
W.~Hackbusch.
\newblock {\em Multi-grid methods and applications}, volume~4.
\newblock Springer Science \& Business Media, 2013.

\bibitem{hestenes1952methods}
M.~R. Hestenes and E.~Stiefel.
\newblock Methods of conjugate gradients for solving linear systems.
\newblock {\em J. Research Nat. Bur. Standards}, 49:409--436, 1952.

\bibitem{HesthavenRozzaStammBook}
J.~S. Hesthaven, G.~Rozza, and B.~Stamm.
\newblock {\em Certified reduced basis methods for parametrized partial
  differential equations}.
\newblock SpringerBriefs in Mathematics. Springer, Cham; BCAM Basque Center for
  Applied Mathematics, Bilbao, 2016.
\newblock BCAM SpringerBriefs.

\bibitem{holmes_lumley_berkooz_1996}
P.~Holmes, J.~L. Lumley, and G.~Berkooz.
\newblock {\em Turbulence, coherent structures, dynamical systems and
  symmetry}.
\newblock Cambridge Monographs on Mechanics. Cambridge University Press,
  Cambridge, 1996.

\bibitem{kadeethum2022enhancing}
T.~Kadeethum, D.~O’malley, F.~Ballarin, I.~Ang, J.~N. Fuhg, N.~Bouklas, V.~L.
  Silva, P.~Salinas, C.~E. Heaney, C.~C. Pain, et~al.
\newblock Enhancing high-fidelity nonlinear solver with reduced order model.
\newblock {\em Sci. Rep.}, 12(1):20229, 2022.

\bibitem{kelley1995iterative}
C.~T. Kelley.
\newblock {\em Iterative methods for linear and nonlinear equations}, volume~16
  of {\em Frontiers in Applied Mathematics}.
\newblock Society for Industrial and Applied Mathematics (SIAM), Philadelphia,
  PA, 1995.
\newblock With separately available software.

\bibitem{LIANG2002527}
Y.~Liang, H.~Lee, S.~Lim, W.~Lin, K.~Lee, and C.~Wu.
\newblock Proper orthogonal decomposition and its applications—{P}art {I}:
  Theory.
\newblock {\em J. Sound Vib.}, 252(3):527--544, 2002.

\bibitem{Mckay2000LHS}
M.~D. McKay, R.~J. Beckman, and W.~J. Conover.
\newblock A comparison of three methods for selecting values of input variables
  in the analysis of output from a computer code.
\newblock {\em Technometrics}, 21(2):239--245, 1979.

\bibitem{nguyen2018reduced}
N.-C. Nguyen and Y.~Chen.
\newblock Reduced-basis method for the iterative solution of parametrized
  symmetric positive-definite linear systems.
\newblock {\em arXiv preprint arXiv:1804.06363}, 2018.

\bibitem{nikolopoulosa2022ai}
S.~Nikolopoulosa, I.~Kalogerisa, G.~Stavroulakisa, and V.~Papadopoulosa.
\newblock Ai-enhanced iterative solvers for accelerating the solution of
  large-scale parametrized systems.
\newblock {\em arXiv preprint arXiv:2207.02543}, 2022.

\bibitem{Quarteroni2015}
A.~Quarteroni, A.~Manzoni, and F.~Negri.
\newblock {\em Reduced basis methods for partial differential equations: {A}n
  introduction}, volume~92.
\newblock Springer, 01 2015.

\bibitem{Rozza2008Reduced}
G.~Rozza, D.~B.~P. Huynh, and A.~T. Patera.
\newblock Reduced basis approximation and a posteriori error estimation for
  affinely parametrized elliptic coercive partial differential equations:
  application to transport and continuum mechanics.
\newblock {\em Arch. Comput. Methods Eng.}, 15(3):229--275, 2008.

\bibitem{saad1986gmres}
Y.~Saad and M.~H. Schultz.
\newblock G{MRES}: a generalized minimal residual algorithm for solving
  nonsymmetric linear systems.
\newblock {\em SIAM J. Sci. Statist. Comput.}, 7(3):856--869, 1986.

\bibitem{tatebe1993multigrid}
O.~Tatebe.
\newblock The multigrid preconditioned conjugate gradient method.
\newblock In {\em NASA. Langley Research Center, The Sixth Copper Mountain
  Conference on Multigrid Methods, Part 2}, 1993.

\bibitem{van1992bi}
H.~A. van~der Vorst.
\newblock Bi-{CGSTAB}: a fast and smoothly converging variant of {B}i-{CG} for
  the solution of nonsymmetric linear systems.
\newblock {\em SIAM J. Sci. Statist. Comput.}, 13(2):631--644, 1992.

\bibitem{zhou2023neural}
M.~Zhou, J.~Han, M.~Rachh, and C.~Borges.
\newblock A neural network warm-start approach for the inverse acoustic
  obstacle scattering problem.
\newblock {\em J. Comput. Phys.}, 490:Paper No. 112341, 16, 2023.

\end{thebibliography}
\begin{appendix}
\section{The estimate of the computational cost of the offline stage}\label{appendix:offline cost}
In this section, we give an estimation of the computational cost of the greedy-based RBM and illustrate the cost challenge of developing high-precision RB solutions. 
Let $\mathcal{M} = \{u(\bmu):\bmu\in \mathcal{D}\}$ denote the solution manifold consisting of all parameter-dependent solutions. Here we first recall an essential definition in the numerical analysis of the RBM, the Kolmogorov $n$-width, indicating the difference between the optimal $n$-dimensional linear approximation space and the solution manifold $\mathcal{M}$
\begin{equation*}
d_n = d_n(\mathcal{M}) = \inf_{{\rm dim}(\mathcal{L})=n}\sup_{u(\bmu)\in \mathcal{M}}{\rm dist}(u(\bmu),\mathcal{L}),
\end{equation*}
where $\mathcal{L}$ denotes the $n$-dimensional linear approximation spaces. Assuming the RB space $W_n$ of dimension $n$ is built by the weak greedy algorithm, we denote the following approximation error
\begin{equation*}
\sigma_n = \sigma_n(\mathcal{M}) = \max_{u(\bmu)\in\mathcal{M}}\Vert u(\bmu)-P_{W_n}u(\bmu)\Vert_V,
\end{equation*}
where $P_{W_n}$ is the projection operator on $W_n$. It was proved in \cite{Binev2011Convergence,DeVore2012Greedy,Buffa2012prioriconvergence} that any polynomial rate of decay achieved by
the Kolmogorov $n$-width $d_n$ is retained by the approximation error $\sigma_n$. Precisely, the following holds, see \cite[Theorem 3.1]{Binev2011Convergence}.
\begin{theorem}\label{Them 1}
For $M>0$ and $s>0$, suppose that $d_n(\mathcal{M})\le M(\max\{1,n\})^{-s}$, $n\ge 0$. We have $\sigma_n(\mathcal{M})\le M_s \gamma^{-2}(\max\{1,n\})^{-s}$, where $M_s = 2^{4s+1}M$, and $\gamma$ is a positive threshold parameter independent of $n$ in the weak greedy algorithm.
\end{theorem}
To achieve a target accuracy $\varepsilon$, with Theorem \ref{Them 1}, an estimate on the number of greedy steps $n(\varepsilon)$ was provided in \cite[Corollary 2.4]{Cohen2020Reduced}, as follows
\begin{corollary}
For $M>0$ and $s>0$, suppose that $d_n(\mathcal{M})\le M(\max\{1,n\})^{-s}$, $n\ge 0$. We have $n(\varepsilon)\le M_1\varepsilon^{-1/s}$, $\varepsilon>0$, where $M_1$ depends on $M$, $s$ and greedy parameter $\gamma$.
\end{corollary}
In the practical calculation, the greedy algorithm is applied based on a discrete training parameter set $\widetilde{\mathcal{D}}$ instead of the continuous parameter set $\mathcal{D}$. In \cite{Cohen2020Reduced}, the authors also estimated the size of $\widetilde{\mathcal{D}}$ under the assumption 
\[
\Vert u_h(\bmu_1)-u_h(\bmu_2)\Vert\le M\Vert\bmu_1-\bmu_2\Vert,\ \bmu_1,\bmu_2\in \mathcal{D},\ M>0.
\]
They showed that the discrete manifold $\widetilde{\mathcal{M}} = \{u(\bmu):\bmu\in \widetilde{\mathcal{D}}\}$ should be a $\varepsilon$-net\footnote{If $\widetilde{\mathcal{M}}$ satisfies 
$\mathcal{M}\subset\bigcup_{u(\bmu)\in \widetilde{\mathcal{M}}} B(u(\bmu),\delta)$,
$\widetilde{\mathcal{M}}$ is called a $\delta$-net of $\mathcal{M}$.} of the manifold $\mathcal{M}$ for the target accuracy $\varepsilon$. Such $\widetilde{\mathcal{M}}$ could be induced by a $M^{-1}\varepsilon$-net of $\mathcal{D}$ that scales like $\#\widetilde{\mathcal{D}}\sim \varepsilon^{-cp}$. The size $\#\widetilde{\mathcal{D}}$ in conjunction with the number of greedy steps $n(\varepsilon)$ shows that the total number of error estimator evaluations is at best of the order $\mathcal{O}(\varepsilon^{-cp/s})$.

The offline cost mainly includes two parts, the cost of computing $n(\varepsilon)$ FOM solutions each of dimension $\N$ and the cost of computing the error estimators based on the RB space of dimension $N=n(\varepsilon)$. Thus the total offline cost scales like
\[
{\rm Poly}(\N)n(\varepsilon)+{\rm Poly}(N)\mathcal{O}(\varepsilon^{-cp/s}).
\]
Given an exponential accuracy $\varepsilon = 10^{-a}$, we can find that the offline cost increases exponentially as the accuracy parameter $a$ and the parameter dimension $p$ increase and the parameter $s$ that indicates the decay rate of the Kolmogorov width decreases.
 
\end{appendix}
\end{document}